\documentclass{article}
 
\usepackage[latin1]{inputenc}
\usepackage[T1]{fontenc}
\usepackage[francais, english]{babel}
\usepackage{lmodern}
\usepackage{stmaryrd}
\usepackage{amsthm}
\usepackage{amsmath}
\usepackage{amssymb}
\usepackage{mathrsfs}
\usepackage{soul}
\usepackage{layout}
\usepackage{setspace}
\usepackage[top=3cm, bottom=3cm, left=3 cm, right=3 cm]{geometry}
\usepackage{pgf,tikz}
\usepackage{graphicx}
\usepackage{xcolor}
\usetikzlibrary{arrows}
\usepackage{listings}
\usepackage{enumitem}
\usepackage{hyperref}
\usepackage{cases}
\usepackage{mathabx}
\usepackage{cleveref}
\usepackage{xcolor}
\usepackage[bottom,flushmargin,hang,multiple]{footmisc}

\newtheorem{lemme}{Lemma}[section]

\newtheorem*{lemme*}{Lemme} 

\newtheorem*{rem.}{Remarque}
\newtheorem*{th*}{Theorem}
\newtheorem*{th.*}{Theorem}

\newtheorem{th.}[lemme]{Theorem}
\newtheorem{theorem}{Theorem}

\theoremstyle{definition}
\newtheorem*{def.}{Definition}

\newcommand{\R}{\mathbb{R}}
\newcommand{\Z}{\mathbb{Z}}

\newcommand{\N}{\mathbb{N}}
\newcommand{\D}{\mathbb{D}}

\newcommand{\supp}{\text{supp }}

\title{Transience in law for  symmetric random walks in infinite measure}
\author{Timothée Bénard}
\date{}

\begin{document}
\maketitle

\bigskip
\abstract{
We consider a random walk on a second countable locally compact topological space endowed with an invariant Radon measure. We show that if the walk is symmetric and if every subset which is invariant by the walk has zero or infinite measure, then one has escape of mass for almost every starting point. We then apply this result in the context of homogeneous random walks on infinite volume spaces, and deduce a converse to Eskin-Margulis recurrence theorem.}

\large

\bigskip

\section{Introduction}

The starting point of this text is an article published by  Eskin and  Margulis in 2004 which studies the recurrence properties of random walks on homogenenous spaces \cite{EM}. The space in question is a quotient  $G/\Lambda$, where   $G$ is a real Lie group and $\Lambda \subseteq G$ a discrete subgroup. Given a probability measure $\mu$ on $G$, we can define a random walk on $G/\Lambda$ with transitional probability measures   $(\mu\ast \delta_{x})_{x\in G/\Lambda}$. In more concrete terms, a random step starting at a point $x\in G/\Lambda$ is performed by choosing an element $g\in G$ randomly with law $\mu$ and letting it act on $x$. The two authors ask about the position of the walk at time $n$ for large values of $n$. They manage to show a surprising result: if $G$ is a simple real algebraic group, if $\Lambda$ has finite covolume in $G$, and if the support of $\mu$ is compact and generates a Zariski-dense subgroup of $G$, then for every starting point $x\in G/\Lambda$, the $n$-th step distribution of the walk   $\mu^{\ast n}\ast \delta_{x}$ does not escape at infinity. More precisely,  all  the weak-$\ast$ limits of $(\mu^{\ast n}\ast \delta_{x})_{n\geq 0}$ have mass $1$. One says there is no escape of mass. This  reminds of the behavior of the unipotent flow as highlighted by Dani and Margulis in \cite{Dan, Margunip},  proving that the trajectories of a unipotent flow on $G/\Lambda$  spend most of their time inside compact sets. Eskin-Margulis' result is actually the starting point of a fruitful analogy with Ratner theorems, that led to the classification of stationary probability measures on $X$ thanks to the work of Benoist and Quint  \cite{BQI, BQII},  followed by 
 Eskin and  Lindenstrauss  \cite{ELII}.

\bigskip

This paper asks the question of a converse to Eskin-Margulis Theorem:

\bigskip
\emph{Is the absence of mass escape characteristic of random walks on  homogeneous spaces of finite volume, or could it also happen for walks in infinite volume? }

\bigskip

Let us illustrate the question with an \emph{example}. Consider $S$  a hyperbolic surface of the form 
\smallskip
$$\rotatebox{-0}{\includegraphics[scale=0.048]{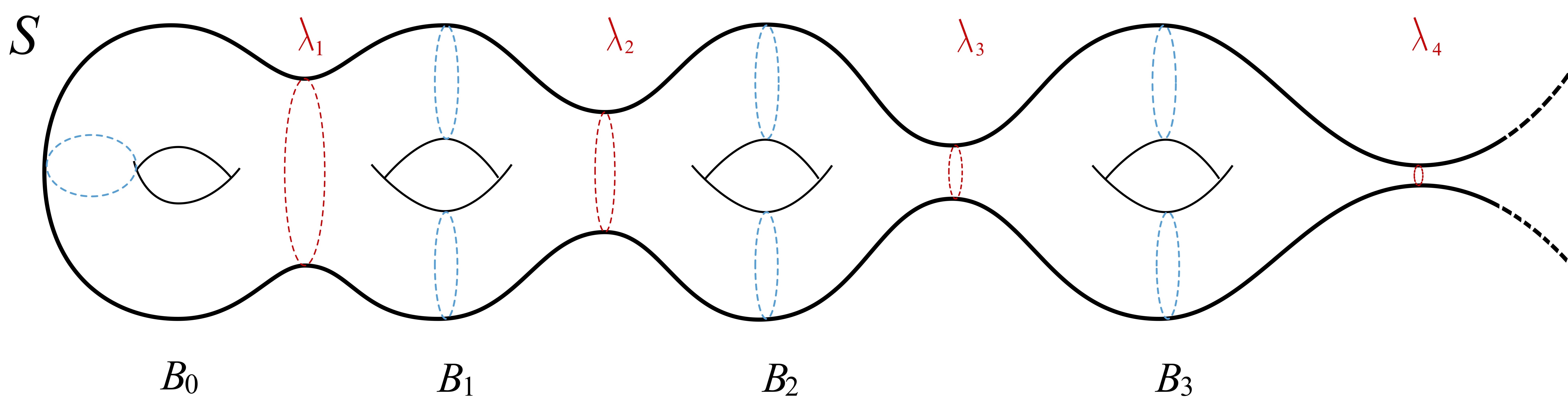}}$$
Such a surface is made of blocks $(B_{i})_{i\geq 0}$ glued together along  geodesic circles (in red) of respective length $(\lambda_{i})_{i\geq 1}\in \R^{\N^*}_{>0}$. Each block comes with a pants decomposition, whose internal boundary components (in blue) are assumed to have length $1$. We consider a (discretized) Brownian motion on (the unit tangent bundle of) $S$ starting from $B_{0}$. If all the $\lambda_{i}$'s are equal, the walk looks like the nearest neighbour random walk on $\N$, so we expect escape of mass. On the other hand, in the degenerate case  where some $\lambda_{i}$ is equal to $0$, the walk evolves in a finite volume space so there is no mass escape by Eskin-Margulis Theorem, or more simply ergodic considerations in this case. Now, we may wonder what happens in intermediate situations where the sequence $(\lambda_{i})_{i\geq 1}$ is positive but allowed to go to zero extremely fast. We will see that  the $n$-th step distribution of the walk  always escapes at infinity, regardless of the choice of  $(\lambda_{i})_{i\geq 1}\in \R^{\N^*}_{>0}$ (\Cref{recip}).

\bigskip

In Section 2, we establish escape of mass in a very general framework, which does not rely on  the algebraic setting mentioned previously.  The measure $\mu$ is assumed to be symmetric, i.e invariant under the inversion map  $g\mapsto g^{-1}$.

\begin{theorem}
 \label{TH1}
Let $X$ be a locally compact second countable topological space equipped with a Radon measure $\lambda$, let $\Gamma$ be a locally compact second countable group acting continuously on $X$ and preserving the measure $\lambda$, let $\mu$ be a probability measure on $\Gamma$ whose support generates $\Gamma$ as a closed group. 

If the probability measure $\mu$ is symmetric and if every measurable $\Gamma$-invariant subset of $X$ has zero or infinite $\lambda$-measure, then for $\lambda$-almost every starting point $x\in X$, one has the weak-$\ast$ convergence: 
$$\mu^{\ast n}\ast \delta_{x} \underset{n \to +\infty}{\longrightarrow} 0$$
\end{theorem}

To put it in a nutshell, a symmetric random walk on a measured space  without finite volume invariant subset is transient in law for almost every starting point.  This result can be seen as an analogue in infinite measure of equidistribution results for random walks in finite measure obtained independently  by Rota \cite{Rota} and Oseledets \cite{Ose}.

In our statement, a measurable subset $A\subseteq X$ is considered as $\Gamma$-invariant if for every $g\in \Gamma$, $\lambda(gA\Delta A)=0$. We will see later an equivalent characterization in terms of the Markov operator of the walk (\Cref{invariance}). 
 
Note also that the condition of symmetry on $\mu$ is necessary: Let $(X, \lambda)$ be a locally compact space with an infinite Radon measure and endowed with a conservative ergodic measure-preserving  $\Z$-action. If $\mu=\delta_{1}$ is the Dirac mass at $1\in \Z$, then for $\lambda$-almost every $x\in X$, the sequence $(n.x)_{n\geq 0}$ comes back close to $x$  infinitely often, so $\mu^{*n}*\delta_{x}=\delta_{n.x}$ cannot weakly converge to $0$. 

Regardless of symmetry assumptions on $\mu$, the proof of \Cref{TH1} still  yields  convergence to $0$ in Cesàro-averages.

\bigskip

\bigskip

In Section 3, we use  \Cref{TH1} to address our original question concerning the  escape of mass of homogeneous  walks on infinite volume spaces. We obtain the following result.

\begin{theorem}  \label{recip}
Let $G$ be a semisimple connected real Lie group with finite center, $\Lambda \subseteq G$ a discrete subgroup of infinite covolume in $G$, and  $\mu$ a probability measure on $G$ whose support generates a group with unbounded projections in the noncompact factors  of $G$.

Then for almost every $x \in G/\Lambda$, one has the weak-$\ast$ convergence:
\begin{align*}
\frac{1}{n} \sum_{k=0}^{n-1}\mu^{\ast k}\ast \delta_{x} \underset{n \to +\infty}{\longrightarrow} 0
\tag{1}\end{align*}

Moreover, if the probability measure $\mu$ is symmetric, then the convergence can be strengthened: 
\begin{align*}
\mu^{\ast n}\ast \delta_{x} \underset{n \to +\infty}{\longrightarrow} 0 \tag{2}
\end{align*}
\end{theorem}

Note  that convergence (1) is sufficient to ensure that  Eskin-Margulis' observations  cannot occur when the quotient $G/\Lambda$ has infinite measure. Indeed, for almost every $x\in G/\Lambda$, we obtain the existence of  an extraction
 $\sigma : \N \rightarrow \N$ such that $$ \mu^{\ast \sigma(n)}\ast \delta_{x} \underset{n \to +\infty}{\longrightarrow} 0$$

\Cref{recip} describes the asymptotic behavior of  the probabilities of position  for  \emph{almost  every} starting point $x\in G/\Lambda$. One may not hope for transience in law for every starting point as it is possible that the orbit $\Gamma.x$ is finite.

\bigskip

To conclude this introduction, we emphasize that our paper focuses on the  behavior \emph{in law} of a random walk on  $G/\Lambda$. A related natural theme of study is the behavior of the walk trajectories for which analogous notions of recurrence or transience exist. Although our conclusions support the idea that walks in infinite volume are always transient in law (the mass escapes),  the picture becomes mixed when it comes to considering walk trajectories. Indeed, as observed in \cite{ConGui} or \cite{Benard21-asympt},  pointwise recurrence or transience also depends on the nature of the ambient space.

\section{A general result of transience in law}

This section is dedicated to the proof of \Cref{TH1}. The proof results from a combination of Dunford-Schwartz Theorem \cite{56Dunford-Schwartz} and Akcoglu-Sucheston's pointwise convergence of alternating sequences \cite{Akc-Suc}. The latter guarantees that for $\lambda$-almost every $x\in X$, the sequence of probability measures $(\mu^{\ast n} \ast \widecheck{\mu}^{\ast n}\ast\delta_{x})_{n\geq 0}$ weak-$\ast$  converges toward a finite measure,  and is based on Rota and Oseledets' original idea to express this alternating sequence in terms of reversed martingales \cite{Rota, Ose}. We  give  a shorter proof than the one  in \cite{Akc-Suc}. Although our proof  follows very closely   the one of  Rota \cite{Rota} who considered  walks on finite volume spaces, we use a different formalism that may be useful to illustrate the technique of ``equidistribution of fibres''  contained in the work of Benoist-Quint \cite{BQII} (see also \cite{Buf}).

\subsection{Backwards martingales}

We first present a convergence theorem for backwards martingales on a $\sigma$-finite measured space. It will play a crucial role in the proof  of the convergence of back-and-forths (\ref{Rotainfini}).

\bigskip

First, let us recall the definition of conditional expectation. 

\smallskip

\begin{def.}[Conditional expectation] \label{d} 
Let $(\Omega, \mathcal{F})$ be a measurable space,  $\mathcal{Q}$ a sub-$\sigma$-algebra of $\mathcal{F}$, and $m$ a positive measure on $(\Omega, \mathcal{F})$ whose restriction $m_{|\mathcal{Q}}$ is $\sigma$-finite. Then, for every function $f \in L^1(\Omega, \mathcal{F}, m)$, there exists a unique function $f' \in L^1(\Omega, \mathcal{Q}, m)$ such that for all $\mathcal{Q}$-measurable subset $A \in \mathcal{Q}$, one has $m(f \,1_{A}) = m(f' \,1_{A})$.  We denote this  function by $\mathbb{E}_{m}(f | \mathcal{Q})$.

\end{def.}

\bigskip

We have the following \cite[page 533]{Jer} (see also \cite{Chow}). 
 \smallskip

\begin{th.}[Convergence of backwards martingales] \label{martdec}

Let $(\Omega, \mathcal{F}, m)$ be a measured space, $(\mathcal{Q}_{n})_{n \geq 0}$ a decreasing sequence of sub-$\sigma$-algebras of $\mathcal{F}$ such that for all $n \geq 0$, the restriction $m_{|\mathcal{Q}_{n}}$ is $\sigma$-finite. Then, for any function $f \in L^1(\Omega,\mathcal{F}, m)$, there exists $\psi \in L^1(\Omega,\mathcal{F}, m)$ such  that we have the almost sure convergence : 
\begin{align*}
\mathbb{E}_{m}(f | \mathcal{Q}_{n})\underset{n\to +\infty}{\longrightarrow} \psi \tag{$m$-a.e.}  
\end{align*}
\end{th.}

\noindent{\bf Remark.} 
If the measure $m$ is $\sigma$-finite with respect to the tail-algebra  $\mathcal{Q}_{\infty}:= \bigcap_{n \geq 0} \mathcal{Q}_{n}$, then  \Cref{martdec} can be deduced from the probabilistic 
case (by restriction to $\mathcal{Q}_{\infty}$-measurable domains of finite measure), and we can certify that $\psi= \mathbb{E}_{m}(f|\mathcal{Q}_{\infty})$. On the extreme opposite, if every  $\mathcal{Q}_{\infty}$-measurable subset of  $\Omega$ has  $m$-measure $0$ or $+\infty$, then, the integrability of  $\psi$ implies that $\psi=0$. The general picture is a direct sum of these two contrasting situations as $\Omega= \Omega_{\sigma} \amalg \Omega_{\infty}$ where $\Omega_{\sigma} $ is a countable union of $\mathcal{Q}_{\infty}$-measurable sets of finite measure,  and the restricted measure $m_{|\Omega_{\infty}}$ takes only the values  $0$ or $+\infty$  on $\mathcal{Q}_{\infty}$ (see \cite{Jer}, footnote of page 533).

\subsection{Convergence of back-and-forths}

We now state and show \Cref{Rotainfini} about the convergence of back-and-forths of the $\mu$-random walk on $X$. We denote by $\widecheck{\mu}:= i_{\ast}\mu$ the image of  $\mu$ under the inversion map $i : \Gamma  \rightarrow \Gamma,\, g \mapsto g^{-1}$.

\bigskip

\begin{th.}[Convergence of  back-and-forths  \cite{Akc-Suc}]\label{Rotainfini}
Let $X$ be a locally compact second countable topological space equipped with a Radon measure $\lambda$, let $\Gamma$ be a locally compact second countable group acting continuously on $X$ and preserving the measure $\lambda$,  and let $\mu$ be a probability measure on $\Gamma$.

There exists a family $(\nu_{x})_{x \in X}$ of  finite measures  on $X$ such that for  $\lambda$-almost every $x \in X$, one has the weak-$\ast$ convergence:
$$(\mu^{\ast n}\ast \widecheck{\mu}^{\ast n}) \ast \delta_{x} \underset{n \to +\infty}{\longrightarrow} \nu_{x}$$
\end{th.}

\bigskip

\begin{proof}

The following proof is inspired by \cite{Rota} and  \cite{BQII}. Denote 
$$B:=\Gamma^{\N^\ast}, \,\,\,\,\,\,\beta := \mu^{\N^\ast}, \,\,\,\,\,\,T: B \rightarrow B, (b_{i})_{i \geq 1} \mapsto (b_{i+1})_{i \geq 1}$$
the one-sided shift. One introduces a $\sigma$-finite fibred  dynamical system  $(B^X, \beta^X, T^X)$ setting
   
\begin{itemize}
\item $B^X:= B\times X$  
\item $\beta^X := \beta \otimes \lambda \in \mathcal{M}^{Rad}(B \times X)$
\item $T^X: B^X \rightarrow B^X, (b,x)\mapsto (Tb, b^{-1}_{1}x)$. 
\end{itemize}

Let $\mathcal{B}$ and $\mathcal{X}$ denote the Borel $\sigma$-algebras of $B$ and $X$. The Borel $\sigma$-algebra of $B^X$ is then the product algebra $\mathcal{B}\otimes \mathcal{X}$. For all $n\geq 0$,  define the \emph{sub-$\sigma$-algebra of the $n$-fibres} of $T^X$ by setting $$\mathcal{Q}_{n}:= (T^X)^{-n}(\mathcal{B}\otimes \mathcal{X})$$
 It is a sub-$\sigma$-algebra of $\mathcal{B}\otimes \mathcal{X}$ such that for all $c\in B^X$, the smallest $\mathcal{Q}_{n}$-measurable subset of $B^X$ containing $c$ is the $n$-fibre $(T^X)^{-n}(T^X)^{n}(c)$. The restriction $\beta^X_{|\mathcal{Q}_{n}}$ is a $\sigma$-finite measure because $\beta^X$ is $\sigma$-finite with respect to the $\sigma$-algebra $\mathcal{B}\otimes \mathcal{X}$ and is preserved by $T^X$. 

\smallskip
As a first step, we will fix a continuous function with compact support  $f \in C^0_{c}(X)$ and  show that   the sequence $\left((\mu^{\ast n}\ast \widecheck{\mu}^{\ast n}\ast \delta_{x})(f)\right)_{n \geq 0}$ converges in $\R$ for $\lambda$-almost every $x$. To this end, we express $(\mu^{\ast n}\ast \widecheck{\mu}^{\ast n}\ast \delta_{x})(f)$ using a conditional expectation and we apply  \Cref{martdec}. Denote  
$$\widetilde{f} : B^X \rightarrow \R, (b, x) \mapsto f(x),\,\,\,\,\,  \varphi_{n}:= \mathbb{E}_{\beta^X}(\widetilde{f} | \mathcal{Q}_{n}) \in L^1(B^X, \mathcal{Q}_{n})$$
 
We first give an explicit formula for the function $\varphi_{n}$. Intuitively, given a point $c=(b,x)\in B^X$, the value $\varphi_{n}(c)$ stands for the mean value  of $\widetilde{f}$ on the smallest  $\mathcal{Q}_{n}$-measurable subset of $B^X$ containing $c$.  By definition, this subset is the $n$-fibre going through $c$ and is identified with the product  $\Gamma^n$ under   the bijection
 $$h_{n,c} : \Gamma^n \rightarrow (T^X)^{-n}(T^X)^{n}(c), \,\,a=(a_{1}, \dots, a_{n}) \rightarrow (aT^nb, \, a_{1}\dots a_{n}b^{-1}_{n}\dots b^{-1}_{1}.x )$$
The following lemma asserts that $\varphi_{n}(c)$ is nothing else than the mean value of  $\widetilde{f}$ on $(T^X)^{-n}(T^X)^{n}(c)\equiv \Gamma^{n}$ with respect to the measure $\mu^{\otimes n}$. 

\bigskip

\begin{lemme} \label{espcond}
Let $n\geq 0$. For $\beta^X$-almost every $(b,x) \in B^X$, one has 
$$\varphi_{n}(b,x)=\int_{\Gamma^n}f(a_{1}\dots a_{n}b^{-1}_{n}\dots b^{-1}_{1}x) \,\,d\mu^{\otimes n}(a)$$
\end{lemme}

\bigskip

\begin{proof}[Proof of \Cref{espcond}]
This result is extracted from  \cite{BQII} (Lemma 3.3). We recall the proof. Up to considering separately the positive and negative parts of $f$, one may assume $f\geq 0$. Denote by $\varphi'_{n} : B^X\rightarrow [0,+\infty]$ the map defined by the right-hand side of the above equation. We show it coincides almost everywhere with $\varphi_{n}$ by proving it also satisfies the axioms for the conditional expectation characterizing $\varphi_{n}$.

As the value $\varphi'_{n} $ at a point $c\in B^X$ only depends on $(T^X)^n(c)$, the map  $\varphi'_{n} $ is $\mathcal{Q}_{n}$-measurable. It remains to show that for every $A\in \mathcal{Q}_{n}$, one has the equality $\beta^X(1_{A}\widetilde{f}) =\beta^X(1_{A}\varphi'_{n})$. Writing $A$ as $A=(T^X)^{-n}(E)$ where $E \in \mathcal{B}\otimes \mathcal{X}$ and remembering that  the measure $\lambda$ is preserved by $\Gamma$, one computes that:

\begin{align*}
\beta^X(1_{A}\varphi'_{n})&=\int_{B\times X\times \Gamma^n}1_{A}(b,x)f(a_{1}\dots a_{n}b^{-1}_{n}\dots b^{-1}_{1}x) \,\,d\mu^{\otimes n}(a)d\beta(b)d\lambda(x) \\
&=\int_{B\times X\times \Gamma^n}1_{E}(T^nb, b^{-1}_{n}\dots b^{-1}_{1}x)f(a_{1}\dots a_{n}b^{-1}_{n}\dots b^{-1}_{1}x) \,\,d\mu^{\otimes n}(a)d\beta(b)d\lambda(x) \\
&=\int_{B\times X\times \Gamma^n}1_{E}(T^nb, x)f(a_{1}\dots a_{n}x) \,\,d\mu^{\otimes n}(a)d\beta(b)d\lambda(x) \\
&=\int_{B\times X}1_{E}(T^nb, x)f(b_{1}\dots b_{n}x) \,d\beta(b)d\lambda(x) \\
&=\int_{B\times X}1_{E}(T^nb, b^{-1}_{n}\dots b^{-1}_{1}x)f(x) \,d\beta(b)d\lambda(x) \\
&=\beta^X(1_{A}\widetilde{f})
\end{align*}
which concludes the proof of \Cref{espcond}.
\end{proof}
\bigskip

\Cref{espcond} implies that for $\lambda$-almost every $x \in X$, 
\begin{align*}
\int_{B}\varphi_{n}(b,x)\,d\beta(b) = (\mu^{\ast n}\ast \widecheck{\mu}^{\ast n}\ast\delta_{x})(f) \tag{$\ast \ast$}
\end{align*}

But  \Cref{martdec} on convergence of backwards martingales  asserts the sequence of conditional expectations $(\varphi_{n})_{n \geq 0}$ converges $\beta^X$-almost-surely. Noticing that  $||\varphi_{n}||_{\infty} \leq ||f||_{\infty}$, the dominated convergence theorem and  equation  $(\ast \ast)$ imply that for $\lambda$-almost every $x\in X$, the sequence $$((\mu^{\ast n}\ast \widecheck{\mu}^{\ast n}\ast\delta_{x})(f))_{n \geq 0}$$ has a limit in $\R$. 

\bigskip
We deduce from the previous paragraph that for $\lambda$-almost every $x\in X$, the sequence of probability measures   $(\mu^{\ast n}\ast \widecheck{\mu}^{\ast n} \ast \delta_{x})_{n \geq 0}$ has a weak-$\ast$ limit (which is a measure on $X$ whose mass is less or equal to one, and possibly null). It is indeed a standard argument,  that uses the separability of the space of continuous functions with compact support on $X$ equipped with  the supremum norm $(C^0_{c}(X),||.||_{\infty})$, and the representation of non negative linear forms on $C^0_{c}(X)$ by Radon measures (Riesz Theorem). This concludes the proof of \Cref{Rotainfini}.

\end{proof}

\subsection{Proof of \Cref{TH1}} \label{proofA}

We now prove \Cref{TH1}, stating that a symmetric random walk on a measured space without finite volume invariant subset is almost everywhere transient in law. The proof will use the \emph{Markov operator} $P_{\mu}$ attached to $\mu$. It acts  on the set of non-negative measurable functions on $X$ via the formula  $$P_{\mu}\varphi (x):=\int_{G}\varphi(gx)\,d\mu(g)$$
and can be extended as a contraction on the spaces $L^p(X,\lambda)$ for  $p\in [1, \infty ]$.

Recall from the introduction that a measurable subset $A \subseteq X$ is \emph{$\Gamma$-invariant} if for all $g \in \Gamma$, one has $\lambda(A \Delta gA)=0$. This condition can be rephrased in terms of the Markov operator: 

\begin{lemme}\label{invariance}
A measurable subset $A \subseteq X$ is \emph{$\Gamma$-invariant} if and only if 
\begin{align*}
P_{\mu}1_{A}=1_{A}\tag{$\lambda$-a.e.}
\end{align*}
\end{lemme}

\begin{proof}The point is to show that $P_{\mu}$-invariance implies $\Gamma$-invariance. Let $A$ be a measurable subset  such that $P_{\mu}1_{A}=1_{A}$ $\lambda$-a.e.  The assumption on $A$ means that  for $\lambda$-almost every $x\in X$, $\mu$-almost every $g\in G$, one has $1_{A}(gx)=1_{A}(x)$. Fubini Theorem then implies  that  for  $\mu$-almost every $g\in \Gamma$, one has $\lambda(A\Delta gA)=0$. The subgroup  $D\subseteq \Gamma$ generated by such elements $g$ is dense in $\Gamma$ and leaves the set $A$ $\lambda$-a.e.-invariant. So we just need to check that the $\lambda$-a.e.-invariance  is preserved by taking limits. Let $g\in \Gamma$, $(g_{n})\in D^\N$ such that $g_{n}\rightarrow g$, let $\varphi\in C^0_{c}(X)$. By dominated convergence,
$$\int_{g_{n}A} \varphi \,d\lambda\, -\, \int_{gA}\varphi \,d\lambda=\int_{A} \varphi(g_{n}.)-\varphi(g.) \, d\lambda  \underset{n \to +\infty}{\longrightarrow} 0  $$
 We deduce that $\int_{A}\varphi \,d\lambda=\int_{gA}\varphi \,d\lambda$. As this is true for every $\varphi\in C^0_{c}(X)$, one concludes that $\lambda(A\Delta gA)=0$. 
 \end{proof}

\bigskip

We can now conclude  the 

\begin{proof}[Proof of \Cref{TH1}]
It is enough  to show that for  $\lambda$-almost every $x \in X$, one has the convergence  $\mu^{\ast 2n}\ast \delta_{x} \rightarrow 0$. According to   \Cref{Rotainfini} and the symmetry of $\mu$, the sequence $(\mu^{\ast 2n}\ast \delta_{x} )_{n \geq 0}$ converges to a finite measure, so it is enough to check the following convergence in average:  for $\lambda$-almost every $x\in X$, 
$$\frac{1}{n}\sum_{k=0}^{n-1}\mu^{\ast n}\ast \delta_{x}  \longrightarrow 0$$

As announced in the preceding remark, we show this last convergence  \emph{without using the assumption of symmetry on $\mu$}. We need to check that for every non-negative  continuous function with compact support  $\varphi \in C^0_{c}(X)^+$, 
\begin{align*} 
\frac{1}{n}\sum_{k=0}^{n-1}P^k_{\mu}\varphi \longrightarrow 0  \tag{ $\lambda$-a.e.}
\end{align*}
 where $P_{\mu}$ denotes the Markov operator of  the walk.

 Dunford-Schwartz Ergodic Theorem \cite{56Dunford-Schwartz, Mey} implies that the sequence of functions $(\frac{1}{n}\sum_{k=0}^{n-1}P^k_{\mu}\varphi)_{n\geq 1}$ converges almost-surely to some function $\psi : X \rightarrow \R_{+}$. As the functions $P^k_{\mu}\varphi$ are uniformly bounded in  $L^2(X, \lambda)$, Fatou lemma implies that $\psi\in L^2(X, \lambda)$.  Furthermore, the function $\varphi$ being bounded, the dominated convergence theorem applied to  the probability space $(\Gamma,\mu)$ gives the $P_{\mu}$-invariance 
\begin{align*}
P_{\mu}\psi =\psi \tag{$\lambda$-a.e.}
\end{align*}

\smallskip
We now infer that $\psi$ is $\Gamma$-invariant, meaning that for $g\in \Gamma$, one has 
the equality $\psi \circ g =\psi$ $\lambda$-a.e. on $X$. To this end, observe that the $P_{\mu}$-invariance of $\psi$ expresses $\psi$ as a barycenter of translates $\psi\circ g$: 
\begin{align*}
\int_{\Gamma}\psi \circ g\, d\mu(g)= \psi \,\,\,\tag{$\lambda$-a.e.}
\end{align*}
 But the functions $\psi \circ g$ all are in $L^2(X,\lambda)$ and have the same norm as $\psi$.  The strict convexity of balls in a Hilbert space then gives  for $\mu$-almost every $g\in \Gamma$,  the equality $\psi \circ g=\psi $ $\lambda$-almost everywhere. As the support of $\mu$ generates $\Gamma$ as a closed subgroup, we infer as in \Cref{invariance} that for all  $g\in \Gamma$, one has $\psi \circ g =\psi$ $\lambda$-a.e., which is the $\Gamma$-invariance announced above.

The $\Gamma$-invariance of $\psi$ implies  that for every constant $c>0$, the set $\{\psi > c\}$ is $\Gamma$-invariant, so has zero or infinite $\lambda$-measure by hypothesis. As $\psi^2$ is integrable, we must have  $\lambda\{\psi >c\}=0$. Finally, we get that $\psi=0$ $\lambda$-almost everywhere, which finishes the proof.

\end{proof}

\section{Application to homogeneous walks on infinite volume spaces}

This section is dedicated to the proof of \Cref{recip}. We let  $G$ be a semisimple connected real Lie group with finite center, $\Lambda \subseteq G$ a discrete subgroup of  infinite covolume in $G$, and $\Gamma\subseteq G$ a closed subgroup.
\bigskip

Let us recall the notion of factors of $G$ used in the section. 

\begin{def.}Denote by $\mathfrak{g}$ the Lie algebra of $G$. It can be uniquely decomposed as a direct sum of simple ideals: $\mathfrak{g} = \mathfrak{g}_{1} \oplus \dots \oplus \mathfrak{g}_{s}$. The \emph{factors} of $G$ are the immersed connected subgroups $G_{1},\dots,G_{s}$ of $G$ whose Lie algebras are $\mathfrak{g}_{1},\dots,\mathfrak{g}_{s}$. They are closed in $G$ and commute mutually: for $i\neq j \in \{1,\dots,s\}$ and $g_{i}\in G_{i}, g_{j}\in G_{j}$ one has $g_{i}g_{j}=g_{j}g_{i}$. Lastly,  the product map $\pi : G_{1} \times \dots \times G_{s} \rightarrow G, (g_{1},\dots,g_{s}) \mapsto g_{1}\dots g_{s}$ is a morphism of groups which is onto and has finite kernel. 
\end{def.}

We  make the assumption that \emph{$\Gamma$ has unbounded projections in the noncompact factors of $G$}, which means that the projection  of $\pi^{-1}(\Gamma)\subseteq G_{1} \times \dots \times G_{s} $ in any $G_{i}$ is unbounded if $G_{i}$ is noncompact.  

\bigskip

 \Cref{recip} expresses  escape of mass for a walk on $G/\Lambda$ induced by a probability measure $\mu$ whose supports generates a dense subgroup of $\Gamma$. We will obtain it as a consequence of \Cref{TH1} together with its comment about the non symmetric case. To apply them, we need to check the assumption that every subset of $G/\Lambda$ which is invariant by the walk has zero or infinite Haar measure. This would be obvious if the action of $\Gamma$ on $G/\Lambda$ were ergodic. However, this is not always the case, even when $\Gamma$, $\Lambda$ are Zariski-dense in $G$.

 \bigskip
 \noindent{\bf Example}.  Denote by $\D$ the Poincaré disk, set $G=PSL_{2}(\R)\equiv \text{Isom}^+(\D)\equiv T^1\D$, and consider a  Schottky subgroup $S_{0}\subseteq G$ whose limit set $\mathscr{L}_{0}$ on the boundary of  $\D$ is contained under four geodesic arcs, which are disjoint and small enough. Set $\Gamma=\Lambda=S_{0}$.  For some non-zero measure subset of unit vectors $x\in T^1\D$, the set $\widebar{x\Lambda}\cap \partial \D=x\mathscr{L}_{0}$ does not intersect the limit set $\mathscr{L}_{0}$ of $\Gamma$. Given such an $x$ and looking in the quotient space, the orbital map $\Lambda \rightarrow \Gamma\backslash G, g\mapsto \Gamma x g $ is proper, so its image cannot be dense.  Thus, the right action of $\Lambda$ on $\Gamma\backslash G$ is not ergodic, or equivalently, the left action of  $\Gamma$ on $G/\Lambda$ is not ergodic.

 \bigskip
 
The absence of finite volume invariant subspaces will be a consequence of Howe-Moore Theorem \cite[Theorem 2.2.20]{Zim} which we now recall.

\begin{th.*}[Howe-Moore]

Let $G$ be a semisimple connected real Lie group with finite center, and $\pi$ a continuous morphism from $G$ to the unitary group of a separable Hilbert space $(\mathcal{H}, \langle.,. \rangle)$.  Assume that every noncompact factor  $G_{i}$ of $G$ has a trivial set of fixed points,
 i.e. $\mathcal{H}^{G_{i}}:=\{x \in \mathcal{H}, \,G_{i}.x=x\}$ is $\{0\}$. 

Then for every $v,w \in \mathcal{H}$, one has 
$$\langle \pi(g).v,w \rangle \underset{g \to \infty}{\longrightarrow} 0$$  

\end{th.*}

In the  statement, the unitary group $U(\mathcal{H})$ is endowed with the strong operator topology, and the notation $g\to \infty$ means that $g$ leaves every compact subset of $G$. 
 \bigskip

Howe-Moore Theorem implies a lemma of rigidity. 
\smallskip

\begin{lemme}\label{rig}
Assume that $G$ has no compact factor. Let $(\mathcal{H}, \rho)$ be a unitary representation of $G$ on a separable Hilbert space. 
$$\text{If  $\,\,\mathcal{H}^G = \{0\}\,\,$ then $\,\,\mathcal{H}^\Gamma = \{0\}$}$$
\end{lemme}

\begin{proof}[Proof of  \Cref{rig}]
Denote by $G_{1}, \dots, G_{s}$ the factors of $G$. Up to pulling back the representation of $G$ by the product map $\pi : G_{1}\times \dots \times G_{s}\rightarrow G, (g_{1}, \dots, g_{s})\mapsto g_{1}\dots g_{s}$,  one may suppose that $G= G_{1}\times \dots \times G_{s}$. 

Assume $s=2$. The hypothesis $\mathcal{H}^G = \{0\}$ implies that $\mathcal{H}^{G_{1}} \cap \mathcal{H}^{G_{2}} =\{0\}$. Thus, we can decompose $$\mathcal{H} = \mathcal{H}^{G_{1}} \oplus \mathcal{H}^{G_{2}} \oplus \mathcal{H}'$$
where $\mathcal{H}'$ is the orthogonal of $\mathcal{H}^{G_{1}} \oplus \mathcal{H}^{G_{2}}$ in $\mathcal{H}$. Moreover, each subspace is invariant by $G$. Let $v \in \mathcal{H}$ be a   $\Gamma$-invariant vector. Decompose $v$ as $v= v_{1} + v_{2} + v'$ with $v_{i} \in \mathcal{H}^{G_{i}}$, $v' \in \mathcal{H}'$. The representation of $G$ on $\mathcal{H}$ leads to a unitary representation of $G_{2}$ on $\mathcal{H}^{G_{1}}$ and the  $\Gamma$ invariance of $v$ implies that $v_{1}$ is invariant under $p_{2}(\Gamma)$, the projection of $\Gamma$ on the factor $G_{2}$. As $p_{2}(\Gamma)$ is unbounded in $G_{2}$, one can apply Howe-Moore Theorem to obtain $v_{1}=0$. In the same way $v_{2}=0$. Thus $v = v' \in \mathcal{H}'$. The representations of  $G_{1}$ and $G_{2}$ induced by  $G$ on $\mathcal{H}'$ have no non-trivial fix point. Hence, we can apply Howe-Moore Theorem one more time to infer that $v'=0$. Finally, $\mathcal{H}^\Gamma =\{0\}$. 

For the general case where $s \geq 1$, argue by induction on $s$ using the previous method and the decomposition of $\mathcal{H}$ as $ \mathcal{H}^{G_{1} \times \dots\times G_{s-1}} \oplus \mathcal{H}^{G_{s}} \oplus \mathcal{H}'$.
\end{proof}

\bigskip

We deduce that for a group $G$ with no compact factor, the action of $\Gamma$ on  $G/\Lambda$ does not have finite volume invariant subspaces. 

\smallskip

\begin{lemme}\label{quasi-erg}
Assume that $G$ has no compact factor. Then every $\Gamma$-invariant subset of $G/\Lambda$ has zero or infinite Haar-measure.
\end{lemme}

\begin{proof}[Proof of \Cref{quasi-erg}]
Argue by contradiction assuming there exists a $\Gamma$ invariant  subset  $A \subseteq G/ \Lambda $  such that  $\lambda(A) \in (0,+\infty)$ for some $G$-invariant Radon measure $\lambda$ on $G/\Lambda$. Consider the regular unitary representation of $G$ on $L^2(G/\Lambda)$, given by the formula $g.f =f(g^{-1}.)$. The characteristic function $1_{A}\in L^2(G/\Lambda)$  is a non-zero fixed point for the action of $\Gamma$. As $G$ has no compact factor,  \Cref{rig} and the assumption on  $\Gamma$ imply there exists a non-zero fixed point $\varphi \in L^2(G/\Lambda)$ for the action of $G$. Such a function is $\lambda$-a.e. constant, implying that $\lambda$ has finite mass. Absurd. 

\end{proof}

We can now conclude with the

\smallskip

\begin{proof}[Proof of \Cref{recip}]
Assume first that the group $G$ has no compact factor. If the probability measure $\mu$ is symmetric, then convergence (2) comes from  \Cref{quasi-erg} and \Cref{TH1}. If there is no assumption of symmetry, we still get the convergence in Cesàro average (1) via the remark following  \Cref{TH1}.

We now explain how to reduce \Cref{recip} to the case 
  where $G$ has no compact factor. Denote by $G_{1}, \dots, G_{s}$ the factors of $G$, and $\pi$  the  induced finite cover of $G$, i.e.  $\pi : G_{1}\times \dots \times G_{s}\rightarrow G, (g_{1}, \dots, g_{s})\mapsto g_{1}\dots g_{s}$. 
There exists a probability measure $\widetilde{\mu}$ on $\Pi_{i=1}^sG_{i}$ whose support is $\pi^{-1}(\supp \mu)$ and such that the $\widetilde{\mu}$-walk on $\Pi_{i=1}^sG_{i}/ \pi^{-1}(\Lambda)$ lifts the $\mu$-walk on $G/ \Lambda$. It is enough to show escape of mass  for this $\widetilde{\mu}$-walk. Denote by $G_{1},\dots, G_{k}$ the non compact factors of $G$  and  $p : \Pi_{i=1}^sG_{i} \rightarrow \Pi_{i=1}^kG_{i}, (g_{i})_{i\leq s}\mapsto (g_{i})_{i\leq k}$ the projection on their product (notice that $k\geq 1$ otherwise $G$ would not have a discrete subgroup of infinite covolume). Then the projection $p(\pi^{-1}(\Lambda))$ is a discrete subgroup of infinite covolume in $\Pi_{i=1}^kG_{i}$. It is enough to prove escape of mass for the projection  $p_{\ast}\widetilde{\mu}$ on $\Pi_{i=1}^kG_{i}$. Note that this probability measure generates a group with unbounded projections in the $G_{i}$'s for $i=1, \dots, k$. Hence, we have reduced \Cref{recip} to the case of a group with no compact factor, which finishes the proof.

\end{proof}

\bibliographystyle{abbrv}

\bibliography{bibliographie}

\end{document}